\documentclass[a4paper,11pt]{scrartcl}
\usepackage[utf8]{inputenc}
\usepackage{todonotes}

\newcommand{\thetitle}{Algebraic structure of aromatic B-series}
 \newcommand{\thekeywords}{Aromatic series, Aromatic trees, Composition Law, Substitution law, Coalgebra of aromatic trees, Coproduct.}
 \newcommand{\theamssubjectclassification}{ %
 37C10, %
 41A58, %
 16T05 %
 }
 \usepackage[utf8]{inputenc}
\usepackage{amsmath,amsfonts,amssymb}
\usepackage{amsthm}
\usepackage{mathtools}
\usepackage{verbatim}
\usepackage{bbm}
\usepackage{tikz}

\usepackage{planarforestloop}
\usepackage{tikz-cd}
\usepackage{enumerate}
\usepackage{ifthen}
\usepackage{xparse}

\usepackage{etrees}
\usepackage{multirow}
\usepackage[section]{placeins}
\usepackage[sortcites,backend=biber,isbn=false,url=false,maxnames=99,giveninits, style=alphabetic]{biblatex}

\newcommand{\theoremname}{Theorem}
\newcommand{\propositionname}{Proposition}
\newcommand{\corollaryname}{Corollary}
\newcommand{\lemmaname}{Lemma}
\newcommand{\definitionname}{Definition}
\newcommand{\remarkname}{Remark}
\newcommand{\exercisename}{Exercise}
\newcommand{\examplename}{Example}

\newcommand{\kk}{\mathbbm{k}}

\newcommand{\from}{\mathpunct{:}}
\newcommand{\RR}{\mathbb{R}}
\newcommand{\CC}{\mathbb{C}}
\newcommand{\AAA}{\mathcal{A}}
\newcommand{\AT}{\mathcal{AT}}
\newcommand{\AF}{\mathcal{AF}}
\newcommand{\F}{\mathcal{F}}
\newcommand{\T}{\mathcal{T}}

\newcommand{\id}{\mathrm{id}}
\newcommand{\union}{\cup}
\newcommand{\V}{\mathcal{V}}

\newcommand{\tensor}{\otimes}

\newcommand{\efor}{\mathbf{1}}
\newcommand{\ud}{\mathrm{d}}
\DeclareMathOperator*{\ddiv}{div}
\newcommand{\lie}[1]{\mathfrak{#1}}

\newcommand{\EE}{\mathcal{E}}

\DeclareMathOperator{\tr}{tr}
\DeclareMathOperator{\Hom}{Hom}

\date{}
\begin{document}
\newtheoremstyle{thmstyle} %
    {\topsep}              %
    {\topsep}              %
    {\itshape}             %
    {}                     %
    {\bfseries}            %
    {.}                    %
    {.5em}                 %
    {\thmname{#1}\thmnumber{ #2}\thmnote{ (#3)}} %
\theoremstyle{thmstyle}
\newtheorem{theorem}{\theoremname}[section]
\newtheorem{proposition}[theorem]{\propositionname}
\newtheorem{corollary}[theorem]{\corollaryname}
\newtheorem{lemma}[theorem]{\lemmaname}
\newtheoremstyle{defstyle}
    {\topsep}
    {\topsep}
    {}
    {}
    {\bfseries}
    {.}
    {.5em}
    {\thmname{#1}\thmnumber{ #2}\thmnote{ (#3)}}
\theoremstyle{defstyle}
\newtheorem{definition}[theorem]{\definitionname}
\newtheorem{remark}[theorem]{\remarkname}
\newtheorem{example}[theorem]{\examplename}
\newtheorem{exercise}{\exercisename}[section]
\title{\thetitle}
\author{Geir Bogfjellmo\\
{\small Department of Mathematical Sciences}\\
{\small Norwegian University of Science and Technology}\\
{\small Trondheim, Norway} }
\maketitle

\abstract{Aromatic B-series are a generalization of B-series. Some of the
  algebraic structures on B-series can be defined analogically for aromatic B-series. This paper derives combinatorial formulas for the composition and substitution laws for aromatic B-series.}


\noindent{\textbf{Mathematics Subject Classification (2010):}} \theamssubjectclassification

\tableofcontents

\section{Introduction}
B-series have long been an important tool for studying numerical integrators
for ordinary differential equations.
These series have an interesting and rich algebraic structure.

A generalization of B-series is formed by the aromatic B-series, introduced by
Munthe-Kaas and Verdier.

The purpose of the present paper is to describe algebraic structures on aromatic
B-series corresponding to known algebraic structures on normal B-series.

Superficially, an aromatic B-series represents a vector field.
It is also interesting to study aromatic S-series, which
can represent scalar functions, vector fields, higher degree differential
operators or linear combinations thereof.\footnote{(Nonaromatic) S-series are
  also defined, but cannot represent non-constant scalar functions.}

In particular:
\begin{itemize}
\item Theorem \ref{thm: compprod} together with Theorem \ref{thm:Sser} describe the composition of aromatic B-series.
\item Theorem \ref{thm: sublaw} describes the substitution law for aromatic B-series.
\end{itemize}

The structure of the paper is as follows:
Section \ref{sec:background} provides a background describing ordinary B-series
and some of their algebraic structure.
Section \ref{sec:basic} defines aromatic forests, aromatic B-series and the generalization aromatic
S-series.
Section \ref{sec:coalgebra} and section \ref{sec:complaw} develop the
combinatorial formulas for composing aromatic S-series and, by extension,
aromatic B-series.
Section \ref{sec:sublaw} studies the substitution law on aromatic B-series and
S-series, as well as the interaction between the substitution and composition
laws.
Section \ref{sec:final} is an application of the substitution law to define a
pseudo-volume-preserving aromatic B-series method.

\section{Background}
\label{sec:background}
The use of formal series as a tool for studying integrators of ordinary
differential equations has a long and successful history,
see e.g.~\cite{MS15} for an overview.

The most commonly known example are the Butcher series or B-series.
These were originally introduced by Butcher \cite{Butcher63,Butcher72} and by
Hairer and Wanner \cite{HW74} as a tool to study Runge--Kutta methods for the
numerical solution of ordinary differential equations.

B-series have a deep algebraic structure, see for instance Manchon
\cite{Manchon15}.
B-series also arise naturally in other fields of mathematics,
Brouder \cite{brouder00} pointed out an important link to the work by Connes and
Kreimer \cite{CK98}, which was originally written in the context of renormalization.

See also \cite[Section III]{HLW} for an introduction to B-series,
and \cite{MMMV17} for a review of the history of B-series.

In the context of numerical integrators, B-series arise as follows:
Let $\kk\in \{\RR, \CC\}$ and
let $f$ be a smooth vector field on a finite dimensional $\kk$-vector space $W$ defining the ODE
\begin{equation}
  \dot{x}=f(x).
  \label{eq: ODE}
\end{equation}
For a large class of one-step integrators (including the
exact ``integrator'' and all Runge--Kutta methods) the following holds:
When the integrator is applied to \eqref{eq: ODE} and the result of one step with the integrator is expanded as a power
series in the step size $h$, all terms occurring in the series expansion are
formed by combining derivatives of $f$, e.g.
\[f' f'''(f'f,f,f)
\]
where $f'=\frac{\ud f}{\ud x}$ is the Jacobian of $f$ and $f'''=\frac{\ud^3
  f}{\ud x^3}$ is the third derivative of $f$, viewed as a trilinear map
$f'''\from W^{\times 3}\to W$.

Such expressions are called \emph{elementary differentials}, and the set of all
elementary differentials is in a one-to-one correspondence to the set of
\emph{non-planar rooted trees}, here denoted $\T$.

For instance, for the elementary differential above,
\begin{equation}
  f' f'''(f'f,f,f)=F_f\left( \forest{b[b[b[b],b,b]]} \right),
\label{eq: F}
\end{equation}
where $F_f$ is the bijective function from $\T$ to the set of elementary
differentials formed from $f$.

For a specific vector field $f$ and tree $\tau$, $F_f(\tau)$ is again vector field
over $W$.

A \emph{B-series} is defined as a formal sum of elementary differentials, that is, a series of
the form
\[
  B_f(a)=\sum_{\tau\in \T} \frac{a(\tau)}{\sigma(\tau)} F_f(\tau),
\]
where $a\from \T\to \kk$, and the normalization factors $\sigma\from \T\to \kk$
are defined such that $\sigma(\tau)$ is equal to the cardinality of the symmetry
group of the tree $\tau$.

A \emph{B-series integrator} is a numerical integrator such that each individual
step of the integrator
can be expanded as a B-series
\[x_{k+1}= x_k +B_{f}(a)(x_k)=[id+B_f(a)](x_k).
\]
(Here the step size $h$ is subsumed into the vector field $f$.)

As mentioned above the class of such method includes all Runge--Kutta methods
and the exact ``method'' defined by exactly following the vector field.

B-series have a rich algebraic structure that has been studied by many authors.
We will briefly cover some of the structure here, see e.g.
\cite{murua99,CHV10,CEM11,Manchon15} for proofs and more background.

The algebraic structure of B-series can be described in terms of certain algebraic objects
called \emph{Hopf algebras.}
See Appendix \ref{app: coalg} for a brief introduction to Hopf algebras.

The \emph{Connes--Kreimer Hopf algebra $H_\F$} is the Hopf algebra over
\emph{rooted forests}, i.e. multisets of rooted trees \cite{CK98}.
It is convenient to extend the definition of $F_f$ to forests,
where the resulting product of elementary differentials is viewed as a
differential operator on smooth functions $C^\infty(W)$,
e.g.
\[F_f(\tau_1\tau_2)[g]=g''\left[ F_f(\tau_1), F_f(\tau_2)\right]\]
and
\[
  F_f(\efor)[g]=g,
\]
where $\efor$ denotes the empty forest.

We denote the algebraic dual of $H_\F$ by $H_\F^\ast$.
The elements of $H_\F^\ast$ are formal series indexed by rooted
forests.

The dual of a Hopf algebra is naturally an associative, unital algebra with a product given by
the dual of the coproduct, known as the \emph{convolution product}:
\begin{equation}
a\cdot b(x)= \mu_\kk\circ (a\tensor b)\circ \Delta(x),
\label{eq: compprod}
\end{equation}
where $\mu_\kk$ denotes multiplication in $\kk$ and $\Delta$ denotes the
coproduct in the Hopf algebra.

For a fixed vector field $f$, $a\in H_\F^{\ast}$ can be identified with a formal
series of differential operators acting on functions in $C^\infty(W)$,
\begin{equation}
 S_f(a)= \sum_{\phi \in \F} \frac{a(\phi)}{\sigma(\phi)} F_f(\phi),
 \label{eq:Snorm}
 \end{equation}
where $\sigma(\phi)$ is the cardinality of the symmetry group of $\phi$.
Such series were named \emph{S-series} by Murua \cite{murua99}.

It turns out that the convolution product in $H_\F^\ast$ corresponds to composition of formal series of differential operators
\[S_f(a_1\cdot a_2)=S_f(a_1)\circ  S_f(a_2).
\]

Of special importance are the \emph{characters} of $H_\F$, that is to say the linear maps
$a\from H_\F \to \kk$ satisfying
\[
 a(xy)=a(x)a(y), \quad \text{and} \quad a(\efor)=1,
\]
where $\efor$ denotes the empty forest.

It can be verified that the set of characters form a group under the
convolution product with identity given by the co-unit in $H_\F$.
We will denote this group by $G(H_\F)$.

There is a natural one-to-one correspondence between B-series and characters of $H_\F$:
For $a\from \T \to \kk$, there is a unique character $\kappa(a)\from H_\F \to
\kk$ that agrees with $a$ on $\T$.

$\kappa(a)$ can be described by formally expanding $g(y+B_f(a)(y))$ as a Taylor
series at $y$.
\[g(y+B_f(a)(y))=S_f(\kappa(a))[g](y).\]

The group of characters is an infinite dimensional Lie group \cite{BS17}, and
its Lie algebra is formed by the \emph{infinitesimal characters} of $H_\F$, that is
morphisms $b\from H_\F\to \kk$, satisfying
\[
  b(xy)=b(x)\epsilon(y)+\epsilon(x)b(y),
\]
where $\epsilon\from H_\F\to \kk$ is the counit of $H_\F$.

It can be verified that the set of infinitesimal characters forms a Lie algebra
under anti-symmetrization of the convolution product.

We will denote the Lie algebra of infinitesimal characters $\lie{g}(H_\F)$.

Infinitesimal characters correspond to formal series of vector fields on $W$.
For example, the modified vector fields of numerical integrators used in
backward error analysis are given by the logarithm associated with the
product \eqref{eq: compprod}, or, equivalently with the Lie group logarithm on
the group of characters.

\subsection{The composition and substitution product}
As we have seen above, composition of two S-series correspond to the convolution
product on $H_\F^\ast$.

There is another product on S-series, the \emph{substitution product.}

The substitution product $\star\from \lie{g}(H_\F)\times H_\F^\ast \to
H_\F^{\ast}$,
arises from replacing the vector field $f$ in an S-series with a B-series $B_f(b)$.
\[S_f(b\star a)= S_{B_f(b)}(a).\]

The substitution product can be described in terms of the
\emph{rooted tree bialgebra} $\tilde{H}$, which is isomorphic to
$H_\F$ as an algebra.
Calaque, Ebrahimi-Fard and Manchon \cite{CEM11} described a compatible coproduct
$\tilde{\Delta}$ on $\tilde{H}$, as well as a left $\tilde{H}$-comodule map $H_\F \to
\tilde{H}\tensor H_\F$.

We will denote the set of characters of $\tilde{H}$ by $G(\tilde{H})$, although
it is only a monoid.
The substitution product $b\star a$ can defined as either:
\begin{enumerate}
  \item The dual of $\tilde{\Delta}$ when both arguments are in $G(\tilde{H})$, or
  \item The dual of the comodule map, when $b\in G(\tilde{H})$, $a\in
    H_\F^\ast.$
\end{enumerate}

The interaction between the substitution and convolution products is defined
via these two descriptions of $\star$, as well as that any character $b \in
G(\tilde{H}$ defines an algebra automorphism $H_\F^\ast\to H_\F^\ast$ by
$a\mapsto b \star a$.

\section{Aromatic B-series}
\label{sec:basic}
McLachlan et.al.~\cite{MMMV15} classified the B-series methods as exactly the integrators
equivariant under all affine maps, including maps between affine spaces of
non-equal dimension.

Munthe-Kaas and Verdier had earlier \cite{MV15} showed that a larger class of
methods, \emph{Aromatic Butcher series methods}, are equivariant  under
all \emph{invertible} affine maps.

Series related to the aromatic Butcher series (aromatic B-series) had been
studied already by
Iserles, Quispel and Tse \cite{IQT07}, and by Chartier and Murua \cite{CM07}.

The crucial difference between B-series and aromatic Butcher series is that in
the aromatic case, \emph{trace operations} are also allowed in forming elementary
differentials, e.g. $\tr(f')f=(\ddiv f) f$.

On the combinatorial side, these new elementary differentials are obtained by
replacing the set of rooted trees with a larger set of directed graphs, e.g.
\[\tr(f')f = F_f(\dloop{b}\forest{b}).\]

For our purposes, a \emph{directed graph} $\gamma=(V,E)$ is defined by a finite set of \emph{vertices or nodes} $V$, and a set of \emph{edges} $E\subseteq V\times V$.
We say that the edge $(v_1, v_2)$ goes out of $v_1$ and into $v_2$.
A \emph{subgraph} of $\gamma$, is another directed graph $(W,F)$ where $W\subseteq V$, $F\subseteq W\times W \cap E$.
In this definition of graph, we allow the empty graph with $0$ vertices, and self-loops.

On a given graph, we define the \emph{direct predecessor} function $\pi_\gamma$ from $V$ to the power set of $V$ by $v_1 \in \pi_\gamma(v_2)$ iff $(v_1,v_2)\in E$.

Two graphs are \emph{equivalent} and we write $\gamma_1=(V_1,E_1)\simeq (V_2,E_2)=\gamma_2$ if there exists a bijection $g\from V_1\to V_2$ such that $(g\times g)(E_1)=E_2$.

The \emph{automorphism group} of a directed graph is the set of permutations $g\from V \to V$ satisfying that $(g\times g)(E)=E$.

Aromatic elementary differentials correspond to equivalence classes of directed graphs satisfying an additional criterion.
\begin{definition}
An \emph{aromatic forest} is an equivalence class of directed graphs where each node has at most one outgoing edge.
We denote the set of aromatic forests as $\AF$.
A \emph{root} of an aromatic forest is a node with zero outgoing edges.
The set of roots of the aromatic forest $\phi$ is denoted $r(\phi)$.
\end{definition}
We will occasionally refer to aromatic forests as graphs (as opposed to equivalence classes of graphs).
Take these statements to mean a member of the equivalence class, the independence of choice of member is either obvious or explicitly stated.

It follows from the definition that an aromatic forest consists of connected
components, each of which has either (a) one root, in which case the connected
component is a \emph{rooted tree}, or (b) no roots, in which case it contains
exactly one cycle and is called an \emph{aroma}.
\begin{definition}
We define the following subsets of $\AF$:
\begin{itemize}
\item $\AAA = \{\efor,\dloop{b}, \dloop{b[b]}, \dloop{b,b},\dloop{b}\dloop{b}, \dloop{b[b[b]},\dotsc  \}$, the set of aromatic forests with no roots.
\item $\AAA' = \{\dloop{b}, \dloop{b[b]}, \dloop{b,b}, \dloop{b[b[b]},\dotsc
  \}$, the set of connected aromatic forests with no roots, or aromas.
\item $\AT = \forest{b}, \forest{b[b]}, \dloop{b}\forest{b}, \forest{b[b,b]}\dotsc$, the set of aromatic forests with exactly one root.
\item $\T=  \forest{b}, \forest{b[b]}, \forest{b[b,b]}\dotsc$, the set  of rooted trees, or connected aromatic forests with exactly one root.
\item $\F$, the set of loopless aromatic forests, which are (unordered) multisets of rooted trees.
The set $\F$ is also called \emph{forests}.
\item $\F_k$, the set of forests with exactly $k$ roots.
\end{itemize}
\end{definition}
We note that $\AF = \AAA \times \F$.

For a fixed vector field $f$ on a finite dimensional vector space $W$, an
aromatic forest defines an \emph{elementary differential operator} acting on
smooth functions on $W$.
\begin{definition}
Let $\phi=(V,E)$ be an aromatic forest, and $f$ a smooth vector field on $W$.
The elementary differential operator $F_f(\phi)$ is an differential operator acting on smooth functions on $W$, defined as follows:
 For each node $p \in V$, form the factor $f^{i_p}_{I_{\pi(p)}}$, where
$I_{\pi(p)}=i_{q_1}i_{q_2}\dotsm$ is the multiindex defined by the direct
predecessor function on $\phi$, that is $q_1,q_2, \dotsc$ are the vertices such
that $(q_j,p)\in E$.
The upper index on $f$ corresponds to the vector components of $f$ and the lower are partial derivatives with respect to the coordinate directions, $f^{i_p}_{i_{q_1}\dotsm i_{q_m}}=\partial^m f^{i_p}/\partial x_{i_{q_1}}\dotsm \partial x_{i_{q_m}}$.
Then, for each root $s$, form the factor $\partial_{i_s} = \partial/\partial{x_{i_s}}$.
\label{def:elemdiff}
Finally, multiply the factors and sum over repeated indices.
\end{definition}

\begin{example}[Elementary differential operator]
\label{exa: eldiffexample}
Let $\gamma$ be the tree with indices
\[\begin{tikzpicture}
\begin{scope}[etree, scale = 1.5]

\placeroots{3}
\children[1]{child{node(i){}}
             child{node(j){}}}
\children[3]{child{node(k){}}}
  \jointrees{1}{2}
  \end{scope}
  \node[above] at (i) {$i$};
  \node[above] at (j) {$j$};
  \node[above] at (k) {$k$};
  \node[below] at (tree1) {$l$};
  \node[below] at (tree2) {$m$};
  \node[below] at (tree3) {$n$};
\end{tikzpicture},
\]
(in this picture, $i,j,k,\dotsc$ are used in place of
$i_p$ for $p \in V(\gamma)$.)
Then $F_f(\gamma)= f^if^jf^l_{ijm} f^m_l f^k f^n_k\partial_n$.
\end{example}
When $\phi$ contains no aromas, the elementary differential operator $F_f(\phi)$
corresponds to the product of elementary differential operators referenced in
section \ref{sec:background}.

In general, an elementary differential operator $F_f(\phi)$ is a differential
operator of degree $|r(\phi)|$.
Thus, if $\gamma \in \AAA$, $F_f(\gamma)$ is a scalar field.
If $\phi = \gamma \tau_1\tau_2\dotsm \tau_r$, where $\gamma \in \AAA$, $\tau_i\in \T$, then
$F_f(\phi)=F_f(\gamma)F_f(\tau_1)\dotsm F_f(\tau_r)$ is the product of the
scalar field $F_f(\gamma)$ and $r$ vector fields $F_f(\tau_1), \dotsc,
F_f(\tau_r)$.

The action of $F_f(\tau)$ on a smooth function $g$ over $W$ is
\[F_f(\phi)[g]=F_f(\gamma) g^{(r)}\left(F_f(\tau_1), \dotsc,
    F_f(\tau_r)\right),\]
where $g^{(r)}$ is the $r$th derivative of $g$.

An elementary differential operator $F_f(\phi)$ can be considered to be a
$|\phi|$-linear function of vector fields $\EE_\phi$ evaluated on $|\phi|$
copies of $f$.
\begin{equation}
  \EE_\phi(f, f,\dotsc, f)=F_f(\phi),
  \label{eq: elemdifffunction}
\end{equation}
where each argument of $\EE_\phi$ is identified with a vertex in the graph $\phi$.

In \cite{MV15}, aromatic B-series methods are defined as integrators whose
series expansions only contains terms of the form $F_{f}(\tau)$ where $\tau$ is
an aromatic tree.

\begin{definition}
Let $a\from \AT\to \kk$ and $f$ be a smooth vector field on the finite-dimensional
vector space $W$.
The \emph{aromatic B-series} $B_f(a)$ is the formal series
\[B_f(a) = \sum_{\tau \in \AT}\frac{a(\tau)}{\sigma(\tau)} F_f(\tau).\]
The normalization constant $\sigma$ is here defined, for an aromatic tree
$\tau$, to be the cardinality of $G_\tau$, the \emph{graph automorphism group}
of $\tau$.

For later use, we also define $\sigma(\phi)$ for an aromatic forest $\phi$ to be
the cardinality of $G_\phi$.

An \emph{aromatic B-series method} is an integrator whose update
map is given by
\[y_{k+1}=y_k +B_{f}(a)(y_k)
\]
\end{definition}
In the formula for the aromatic B-series method, the step size parameter $h$ is
subsumed into the vector field $f$.

To be able to compose aromatic B-series, we want to describe the effect of such
an integrator to a function, i.e. evaluate $g(y+B_f(a)(y))$ for smooth $g$. To
do this, we introduce aromatic S-series, which mirror normal S-series \eqref{eq:Snorm}.

\begin{definition}
Let $a\from \AF \to \kk$, and $f$ be a smooth vector field on a finite
dimensional vector space $W$.
Let the \emph{aromatic S-series} of $a$ be the formal series of differential
operators acting on smooth functions $g\in C^\infty(W)$ defined by
\[S_f(a)[g]=\sum_{\phi \in \AF} \frac{a(\phi)}{\sigma(\phi)}F_f(\phi)[g].\]
\label{def:Sseries}
\end{definition}

For an aromatic \emph{B-series} $B_f(a)$, where $a\from \AT\to \kk$, there is a
corresponding aromatic \emph{S-series} $S_f(\kappa(a))$.
The action of $S_f(\kappa(a))$ on a smooth function $g$, is calculated by expanding
$g(y+B_f(a)(y))$ as a Taylor series around $y$.
Obtaining the coefficients of $\kappa(a)$ will require some algebraic machinery which will be developed in the next section.

For now, we define the composition product of two S-series.
\begin{definition}
For $a,b \from \AF \to \kk$, define the \emph{composition product} $a\cdot b$ be the aromatic B-series defined by its action on smooth functions $S_f(a\cdot b)[g] = S_f(a)[S_f(b)[g]].$
\end{definition}

The building block of the composition product is the composition of two elementary differential operators.
\begin{definition}
The \emph{composition} of two aromatic forests $\phi_1 \circ \phi_2$ is the formal sum of aromatic forests defined as follows: Let $\star$ be a dummy node not contained in either $V(\phi_1)$ or $V(\phi_2)$  and $(V(\phi_2) \union \star)^{r(\phi_1)}$ the set of functions $\rho\from r(\phi_1)\to V(\phi_2)\union \star$. Let
\[\phi_1 \circ \phi_2 = \sum_{\rho\in (V(\phi_2)\union \star)^{r(\phi_1)}} \Phi(\phi_1, \phi_2, \rho),\]
where $\Phi(\phi_1, \phi_2, \rho)$ is the aromatic forest with nodes $V(\phi_1)\union V(\phi_2)$ and edges $E(\phi_1)\union E(\phi_2) \union E_\rho$, where
$E_\rho= \{(p, \rho(p))\text{ s.t. }p\in r(\phi_1), \rho(p)\neq \star \}$.
\label{def:products}
\end{definition}
Informally, the composition product is formed by the sum over all possible ways to add edges from some (possibly none) roots of $\phi_1$ to the nodes of $\phi_2$.

\begin{example}
  In this example, the nodes are colored to more clearly separate the two
  factors.
  \[
    \begin{aligned}
      \forest{w[w],w}\circ \dloop{b[b]}=&
      \dloop{b[b]}\forest{w[w],w}+\dloop{b[b[w]]}\forest{w[w]}+\dloop{b[w,b]}\forest{w[w]}\\
      &+\dloop{b[b[w[w]]]}\forest{w}+\dloop{b[w[w],b]}\forest{w}+\dloop{b[b[w[w],w]}+\dloop{b[w[w],b[w]]}
     \end{aligned}
  \]

\end{example}

\begin{lemma}
When $g$ is a smooth function over $W$, the equality
\[F_f(\phi_1)[F_f(\phi_2)[g]]=F_f(\phi_1 \circ \phi_2)[g],\]
holds.
\label{lem:compprod1}
\end{lemma}
\begin{proof}
By Definition \ref{def:elemdiff} $F_f(\phi_{1})$ is a sum of products of the form
\[F_f(\phi_1)= \prod_{p\in V(\phi_1)} f^{i_p}_{I_{\pi(p)}} \prod_{s\in r(\phi_1)} \partial_{i_s},\]
and similar for $\phi_2$.
Use the product rule to expand
\[F_f(\phi_1)[F_f(\phi_2)[g]]=  \prod_{p\in V(\phi_1)} f^{i_p}_{I_{\pi(p)}} \prod_{s\in r(\phi_1)} \partial_{i_s}\left[ \prod_{q\in V(\phi_2)} f^{i_q}_{I_{\pi(q)}} \prod_{t\in r(\phi_2)} g_{i_t}\right].\]
The result is a sum of terms, where in each term, each of the operators $\partial_{i_s}$ is applied to (a) one of the factors $f^{i_q}_{I_{\pi(q)}}$ or (b) one of the factors $g_{i_t}$.
Referring to the definition of the elementary differential operators and the composition, case (a) corresponds to adding an edge $s\rightarrow q$, while (b) corresponds to not adding an edge from $s$.
The sum is then taken over all possible choices.
\end{proof}

\section{Coalgebra structure}
\label{sec:coalgebra}
We will be interested in series indexed by aromatic forests, and composition of
such series.

A convenient environment for studying series indexed by elements of a certain set $S$ is the dual space of the free vector space over $S$.
Let
\begin{equation}
  C_{\AAA} = \bigoplus_{\gamma \in \AAA} \kk \gamma,\ H_{\F} = \bigoplus_{\omega \in \F} \kk \omega \text{ and }C_{\AF} =C_{\AAA} \otimes H_{\F} = \bigoplus_{\phi \in \AF} \kk \phi,
  \label{eq: spaces}
\end{equation}
 denote the free vector spaces over the sets $\AAA$, $\F$ and $\AF$.
$H_{\F}$, when equipped with the correct structure, is the
\emph{Connes--Kreimer} Hopf algebra from B-series theory.

$C_{\AAA}$ is the vector space of commutative polynomials in the variables $\AAA'$.

The algebraic dual spaces of the spaces in \eqref{eq: spaces} are
\begin{equation}
C_{\AAA}^\ast = \prod_{\gamma \in \AAA} \kk \gamma^\ast,\ H_{\F}^\ast = \prod_{\omega \in \F} \kk \omega^\ast \text{ and }C_{\AF}^\ast = \prod_{\phi \in \AF} \kk \phi^\ast,
  \label{eq: dualspaces}
\end{equation}
where $\gamma^\ast, \omega^\ast$ and $\phi^\ast$ denote elements in the
respective dual bases.

An element $a\in C_{\AAA}^\ast$ is identified with the formal series
\[\sum_{\gamma\in\AAA} \frac{a(\gamma)}{\sigma(\gamma)}\gamma.\]
Correspondingly for elements in $H_{\F}^\ast$ and $C_{\AF}^\ast$.

In the following, we will describe coalgebra\footnote{All (co)algebras considered in this paper are (co)associative and (co)unital, and we will omit the qualifiers.} structures on $C_\AAA$, $H_{\F}$ and $C_{\AF}$. We recall that a coalgebra is a vector space $C$ equipped with a coproduct
$\Delta \from C \to C\tensor C$ and a counit $\epsilon \from C \to \kk$.
Following Sweedler \cite{Swe69}, the coproduct is written
$\Delta(c)=\sum_{(c)}c_{(1)}\tensor c_{(2)}.$ For a precise definition of a
coalgebra, see Appendix \ref{app: coalg} or \cite{Swe69}.

$C_\AAA$ is the dual of the algebra of formal power series in the variables
$\AAA'$.
Let $a\in C_\AAA^\ast$ be identified with the formal infinite series $\sum_{\gamma\in \AAA} \frac{a(\gamma)}{\sigma(\gamma)}\gamma$
and correspondingly for $b\in C_\AAA^\ast$. Let $a\cdot b\in C_\AAA^\ast$ be defined by
\[\sum_{\gamma\in \AAA} \frac{a\cdot b(\gamma)}{\sigma(\gamma)}\gamma= \left(\sum_{\gamma_1} \frac{a(\gamma_1)}{\sigma(\gamma_1)}\gamma_1 \right)\left(\sum_{\gamma_2} \frac{b(\gamma_2)}{\sigma(\gamma_2)}\gamma_2 \right),\]
where the right hand side is the formal multiplication of series.

It is clear that $a\cdot b(\gamma)$ is a finite sum of terms of the form $a(\gamma_{(1)})b(\gamma_{(2)})$, and we define the coproduct $\Delta_{\AAA}(\gamma)=\sum_{(\gamma)}\gamma_{(1)}\otimes \gamma_{(2)}$ by
\[a\cdot b(\gamma) = \sum_{(\gamma)} a(\gamma_{(1)})b(\gamma_{(2)}).\]
Coassociativity of $\Delta_{\AAA}$ follows from associativity of formal commutative multiplication of series.
The counit $\epsilon_{\AAA}$ is given by $\efor^* \from C_\AAA \to \kk$.

\begin{example}
\begin{multline*}
      \Delta_\AAA(\dloop{b}\dloop{b}\dloop{b,b})=\\
      \efor\tensor\dloop{b}\dloop{b}\dloop{b,b}+2\dloop{b}\tensor
\dloop{b}\dloop{b,b}+\dloop{b}\dloop{b}\tensor\dloop{b,b}+2\dloop{b}\dloop{b,b}\tensor\dloop{b}+\dloop{b,b}\tensor\dloop{b}\dloop{b}+\dloop{b}\dloop{b}\dloop{b,b}\tensor
\efor.
\end{multline*}
\end{example}

The coproduct of $H_\F$ is the dual of the Grossman--Larson product as defined in \cite{CK98}.
Both $C_{\AAA}$ and $H_\F$ are graded connected coalgebras where the grading is induced by the number of vertices of the graphs, and $(C_{\AAA})_0 = (H_\F)_0 =\kk \efor \simeq \kk$.

Since $C_{\AF} = C_{\AAA} \otimes H_{\F}$, we can identify
\[(C_{\AAA} \otimes H_{\F})^\ast \simeq \Hom_{\kk}(H_{\F}, C_{\AAA}^\ast),\]
where $\Hom_{\kk}(H_{\F}, C_{\AAA}^\ast)$ denotes linear functions from $H_{\F}$ to $C_{\AAA}^\ast$, in the usual way: $a\from C_{\AAA} \otimes H_{\F} \to \kk$ is mapped to $a^\top \from H_{\F} \to  C_{\AAA}^\ast$ by
\begin{equation}
\langle a^\top(\phi), c \rangle = a(c\tensor \phi).
\label{eq:xlaw}
\end{equation}
$H_\F$ with the commutative concatenation product is an algebra, while $C_{\AAA}^\ast$ is
the algebra of series in the variables $\AAA^\prime$.

In other words, series indexed by aromatic forests can be identified with linear
maps between algebras $H_{\F}\to C_{\AAA}^\ast$.
We will see later that series corresponding to algebra morphisms $H_{\F}\to
C_{\AAA}^\ast$, or $C_{\AAA}$\emph{-valued characters of }$H_{\F}$, play a special role.

\begin{definition}
An \emph{admissible partition} $p_\phi$ of an aromatic forest $\phi$ is a partition of the graph $\phi$ into two (possibly empty) subgraphs $R(p_\phi)$ and $P^\ast(p_\phi)$, such that no edge in $\phi$ go from $R(p_\phi)$ to $P^\ast(p_\phi)$.
\end{definition}
The graded coproduct is defined on aromatic forests as
\begin{equation}
\Delta_{\AF}(\phi) = \sum_{(\phi)}\phi_{(1)} \tensor \phi_{(2)} =  \sum_{p_\phi} P^\ast(p_\phi) \tensor R(p_\phi),
\label{eq:coprod}
\end{equation}
and extended linearly to $C_\AF$.

$\Delta_\AF$ is calculated for small aromatic forests in Appendix \ref{app: tables}.

The counit is defined by $\epsilon(\phi)=1$ if $\phi=\efor$, $\epsilon(\phi)=0$ otherwise.
\begin{theorem}
$(C_{\AF}, \Delta_\AF, \epsilon)$ is a graded coalgebra,
\end{theorem}
\begin{proof}
 That $\Delta_\AF$ is coassociative follows from the observation that $(\Delta\tensor \id) \circ \Delta$ and $(\id\tensor \Delta) \circ \Delta$ both can be interpreted as the splitting of $\phi$ in three parts $\phi_1, \phi_2, \phi_3$ such that no edges goes from $\phi_3$ to $\phi_2$ nor from $\phi_2$ to $\phi_1$.
The counital property follows from
\[\sum_{p_\phi} \epsilon(P^*(p_\phi))R(p_\phi) = \epsilon(\efor)\phi = \phi,\]
and
\[\sum_{p_\phi} P^*(p_\phi)\epsilon(R(p_\phi)) = \phi \epsilon(\efor) = \phi.\]

Finally, the grading is induced by the number of nodes of $\phi$.
\end{proof}
We note that as a vector space $C_\AF = C_\AAA \tensor H_\F$, but $\Delta_\AF$
is \emph{not} the coproduct inherited from this product.

$C_{\AF}$ can be made into an commutative, but not co-commutative Hopf algebra by adjoining the commutative concatenation
product $\sqcup$.
However, this structure is not as closely connected to aromatic B-series integrators as
the Connes--Kreimer Hopf algebra is to B-series integrators, as will be seen in
section \ref{sec:complaw}.

\begin{theorem}
$(C_{\AF}, \sqcup, \efor, \Delta_{\AF}, \epsilon),$ is a graded, connected, commutative
bi{-}algebra, thus also a commutative Hopf algebra.
\label{thm:AFHopf}
\end{theorem}
\begin{proof}
The compatibility conditions are straight-forward to check, and the existence of
an antipode follows from Theorem \ref{thm: bihopf}.
\end{proof}

\section{The composition law}
\label{sec:complaw}

We are now ready to tackle the composition of aromatic S-series.
Recall that for normal S-series, the composition corresponds to the dual of the
coproduct in $H_\F$.

The coproduct in $C_{\AF}$ has been designed so that this correspondence also
holds between aromatic S-series and the coproduct $\Delta_{\AF}$.

In Definition \ref{def:Sseries}, an aromatic S-series was defined by a function
$a\from \AF\to \kk$.
By a slight misuse of notation we extend $a$ linearly to a linear function
$a\from C_\AF\to \kk$.

\begin{theorem}
  Let $a,b \in C_\AF^\ast$, $f$ a smooth vector field on $W$, and $g$ a
  smooth function on $W$.
Then
\begin{equation}
S_f(b)[S_f(a)[g]] =S_f(b\cdot a)[g],
\label{eq:compprodS}
\end{equation}
where
\begin{equation}
b\cdot a(\phi) = (b\tensor a)(\Delta_{\AF}(\phi))= \sum_{(\phi)} b(P^*(p_\phi))a(R(p_\phi)).
\label{eq:compprod}
\end{equation}
\label{thm: compprod}
\end{theorem}

\begin{proof}
The composition of two aromatic S-series $S_f(b)\circ S_f(a)$ is a series of compositions of
elementary differential operators.

Using Lemma \ref{lem:compprod1}, the composition of two elementary differential
operators becomes a finite sum of elementary differential operators.
Thus, the composition of two aromatic S-series is an aromatic S-series.

What needs to be proven is that this aromatic S-series is equal to the S-series
$S_f(b\cdot a)$, where $b\cdot a$ is given by the convolution product \eqref{eq:compprod}.

The technique used in the proof is to write both sides of \eqref{eq:compprodS} as
a \emph{triple} sum over aromatic forests, where the coefficients are non-zero
if and only if the third forest $\psi$ occurs as a term in the composition of the first two
$\phi_2\circ \phi_1$, or
equivalently that $\phi_1\tensor\phi_2$ occurs as a term in $\Delta_{\AF}(\psi)$.
Finally, the non-zero terms are shown to be equal by a use of the orbit-isotropy
theorem \cite[Theorem 16.16]{Fraleigh}

We start with the composition $S_f(b)[S_f(a)[g]$:
\[S_f(b)[S_f(a)[g]] = \sum_{\phi_2\in \AF} \frac{b(\phi_2)}{\sigma(\phi_2)} F_f(\phi_2)\left[\sum_{\phi_2\in \AF} \frac{a(\phi_1)}{\sigma(\phi_1)} F_f(\phi_1)[g]\right].\]
Using Lemma \ref{lem:compprod1}, we rewrite
\begin{equation}
\begin{aligned}
S_f(b)&[S_f(a)[g]] = \sum_{\phi_2\in \AF} \frac{b(\phi_2)}{\sigma(\phi_2)} F_f(\phi_2)\left[\sum_{\phi_1\in \AF} \frac{a(\phi_1)}{\sigma(\phi_1)} F_f(\phi_1)[g]\right]\\
                  &= \sum_{\phi_2\in \AF}\sum_{\phi_1\in \AF}\frac{b(\phi_2)a(\phi_1)}{\sigma(\phi_2)\sigma(\phi_1)}  F_f(\phi_2\circ \phi_1)[g] \\
                  &= \sum_{\phi_2\in \AF}\sum_{\phi_1\in \AF} \sum_{\rho\in (V(\phi_1)\union \star)^{r(\phi_2)}} \frac{b(\phi_2)a(\phi_1)}{\sigma(\phi_1)\sigma(\phi_2)} F_f(\Phi(\phi_2, \phi_1, \rho))[g]\\
                  &=\sum_{\phi_1\in \AF}\sum_{\phi_2\in \AF} \sum_{\psi\in \AF} \frac{M_1(\phi_1, \phi_2, \psi)}{\sigma(\phi_1)\sigma(\phi_2)}b(\phi_2)a(\phi_1) F_f(\psi)[g],
\end{aligned}
\label{eq:compprod1}
\end{equation}
where $M_1(\phi_1, \phi_2, \psi)$ is the multiplicity of $\psi$ in
\[\phi_2 \circ \phi_1 = \sum_{\rho\in (V(\phi_1)\union
    \star)^{r(\phi_2)}}\Phi(\phi_2, \phi_1, \rho).\]

We continue with $S_f(b\cdot a)[g]$:
\begin{equation}
\begin{aligned}
S_f(b\cdot a)[g] &=\sum_{\psi \in \AF} \frac{b\cdot a(\psi)}{\sigma(\psi)}F_f(\psi)[g] \\
                 &=\sum_{\psi \in \AF}  \sum_{(\psi)} \frac{b(\psi_{(1)})a(\psi_{(2)})}{\sigma(\psi)}F_f(\psi)[g]\\
                 &= \sum_{\phi_1\in \AF}\sum_{\phi_2\in \AF} \sum_{\psi\in \AF} \frac{M_2(\phi_1, \phi_2,\psi)}{\sigma(\psi)} b(\phi_1)a(\phi_2)F_f(\psi)[g],
\label{eq:compprod2}
\end{aligned}
\end{equation}
where $M_2(\phi_1, \phi_2,\psi)$ is the multiplicity of $\phi_1\tensor \phi_2$ in $\Delta_C(\psi)$.

The claim is true if
\[\frac{M_1(\phi_1, \phi_2, \psi)}{\sigma(\phi_1)\sigma(\phi_2)}=\frac{M_2(\phi_1, \phi_2,\psi)}{\sigma(\psi)},\]
for all $\phi_1,\phi_2,\psi$.
We proceed by proving this equality.

(1) $M_1=0 \Leftrightarrow M_2=0$: If edges can be added from $\phi_2$ to $\phi_1$ to form $\psi$, the splitting of $\psi$ into $\phi_1\tensor \phi_2$ is an admissible partition and vice versa.

(2) If $M_1,M_2 \neq 0$, $\frac{\sigma(\phi_1)\sigma(\phi_2)}{M_1(\phi_1, \phi_2, \psi)}$ and
$\frac{\sigma(\psi)}{M_2(\phi_1, \phi_2,\psi)}$ are equal to the cardinality of subgroups of, respectively, $G_{\phi_1}\times G_{\phi_2}$ and $G_\psi$ and we will prove that these subgroups are isomorphic.

For the remainder of the proof, let $\phi_1, \phi_2, \psi,\rho', p'_\psi$ be such that
\[\Phi(\phi_2, \phi_1, \rho')=\psi\quad\text{and}\quad  P^{\ast}(p'_\psi) \tensor R(p'_\psi)= \phi_1 \tensor \phi_2.\]

Let $G_{\phi_1}\times G_{\phi_2}$ act on $(V(\phi_1)\union \star)^{r(\phi_2)}$ from the left by setting
\[(g_1, g_2)\cdot \rho (p)= g_1 \circ \rho \circ g_2^{-1}(p),\]
where we define $g_1(\star)=\star$.

Now, $\Phi(\phi_2, \phi_1, \rho') \simeq \Phi(\phi_2, \phi_1, (g_1, g_2)\cdot \rho')$ as graphs, and $M_1(\phi_1, \phi_2, \psi)$ is the size of the orbit of $\rho'$ under the action of $G_{\phi_1}\times G_{\phi_2}$.
The orbit-isotropy theorem yields that $\frac{\sigma(\phi_1)\sigma(\phi_2)}{M_1(\phi_1, \phi_2, \psi)}$ is the size of the isotropy group of $\rho'$ under the same action.
We call this group $H_1$.

The admissible partition of $\psi$, $p_\psi$ is fully defined by the set of vertices forming the subgraph $R(p_\psi)$.
We let $G_\psi$ act on the set of admissible partitions of $\psi$ by its action on the subgraph $R(\psi)$.
(That is, by considering the simultaneous action on the set of vertices $R(\psi)$.)
Since $G_\psi$ acts on $\psi$ by graph isomorphisms, the restriction of the action to a subgraph is a graph isomorphism of the subgraph.
Therefore, $R(g\cdot p'_\psi) \simeq R(p'\psi)$ and $P^\ast(g\cdot p'_\psi) \simeq P^\ast(p'\psi)$ as graphs, and $M_2(\phi_1,\phi_2,\psi)$ is the size of the orbit of $p'_\psi$ under the action of $G$.
The orbit-isotropy theorem yields that $\frac{\sigma(\psi)}{M_2(\phi_1, \phi_2,\psi)}$ is the size of the isotropy subgroup of $p'_\psi$ under the same action. We call this group $H_2$.

We now prove that $H_1 \simeq H_2$ as groups.
Let $(g_1,g_2)\in H_1$ i.e. $g_1 \circ \rho' = \rho' \circ g_2\lvert_{r(\phi_2)}$.
Now, define the action of $(g_1,g_2)$ on $\psi=\Phi(\phi_2, \phi_1, \rho')$ by letting $g_1,g_2$ act on the subgraphs $\phi_1,\phi_2$.
This is a graph automorphism of $\psi$ because for $(r, \rho'(r))\in E_\rho'$, $(g_2(r), g_1\circ \rho'(r))=(g_2(r),\rho' \circ g_2(r))\in E_\rho'$.
It is in $H_2$ because  $(g_1,g_2)$ maps $R(p'_\psi) \simeq \phi_1$ to $R(p'_\psi)$.

The converse isomorphism $H_2 \to H_1$ follows from considering $g\in H_2$ and doing the appropriate restrictions to get $(g_1,g_2)\in G_{\phi_1}\times G_{\phi_2}$.
For $(p, \rho(p))\in E(\psi)$ , we get $(g_2(p), g_1\circ \rho(p)) \in E(\psi)$.
As $g_2(p)\in \phi_2$, $g_1\circ \rho(p)\in \phi_1$, we must have $g_1 \circ \rho(p)=\rho \circ g_2(p)$.
Therefore $(g_1,g_2)\in H_1$.

As $H_1 \simeq H_2$, we get
\[\frac{\sigma(\phi_1)\sigma(\phi_2)}{M_1(\phi_1, \phi_2,
    \psi)}=|H_1|=|H_2|=\frac{\sigma(\psi)}{M_2(\phi_1, \phi_2,\psi)},\]
and the claim follows.
\end{proof}

Theorem \ref{thm: compprod} says that the composition product $C_{\AF}^\ast \tensor C_{\AF}^\ast \to C_{\AF}^\ast$ is the dual of the coproduct in $C_{\AF}$.
As the dual space of a graded connected coalgebra, the invertible elements of $C_{\AF}^\ast$ are $\{a\in C_{\AF}^\ast \text{ s.t. } a(\efor)\neq 0\}$,
and the inverse is given recursively by
\[a^{-1}(\efor)=1/a(\efor), \qquad a^{-1}(\phi)=-\frac{a(\phi)+\sum_{(\phi)}a(\phi')a^{-1}(\phi'')}{a(\efor)},\]
where $\sum_{(\phi)}\phi'\tensor \phi''= \Delta(\phi)-\efor \tensor \phi - \phi \tensor \efor$ is the reduced coproduct.

We proceed by describing the aromatic S-series arising from aromatic B-series
integrators,
that is, the S-series that correspond to formally expanding $g(y+B_f(a)(y))$ as
a Taylor series around $y$.

Recall that for ordinary B-series, these correspond to $\kk$-valued characters of $H_\F$.

However for aromatic B-series, the corresponding elements turn out not to be
characters of $C_{\AF}$ with the natural algebra structure described in Theorem
\ref{thm:AFHopf}.
Rather, they are algebra morphisms $H_\F\to C_\AAA^{\ast}$.

The proof closely follows a similar proof for B-series in \cite{murua99}.

\begin{theorem}
Extend $a\from C_\AT \to \kk$ to a function $\kappa(a)\from C_\AF \to \kk$ by requiring that $\kappa(a)$ and $a$ coincide on $\AT$, and that $\kappa(a)^\top\from H_{\F} \to  C_{\AAA}^\ast$, defined as in \eqref{eq:xlaw}, is an algebra morphism.
Then $g(y+B_f(a)(y))= S_f(\kappa(a))[g](y)$.
\label{thm:Sser}
\end{theorem}

\begin{proof}
We have that
$\AT= \AAA \times \T$, so we can rewrite the sum
\[
B_f(a) =\sum_{\tau \in \AT} \frac{a(\tau)}{\sigma(\tau)} F_f(\tau) = \sum_{\theta \in \T}\sum_{\gamma\in \AAA} \frac{a(\gamma\theta)}{\sigma(\gamma\theta)} F_f(\gamma\theta).
\]
Now, we make two observations: (i) $F_f(\gamma\theta)= F_f(\gamma)F_f(\theta)$. (ii) The automorphism group of $\gamma\theta$ cannot mix connected components with 0 and 1 root, so $\sigma(\gamma\theta)=\sigma(\gamma)\sigma(\theta)$. Collecting terms, we can write
\[
B_f(a) = \sum_{\theta \in \T} \left[ \sum_{\gamma\in \AAA } \frac{a(\gamma\theta)}{\sigma(\gamma)} F_f(\gamma) \right]\frac{1}{\sigma(\theta)} F_f(\theta).
\]
The sum inside the brackets is a formal series of scalar functions and can be expressed as series indexed by the aromas $\AAA$:
\[S_f(a^\top(\theta))=\sum_{\gamma \in \AAA}
  \frac{a(\gamma\theta)}{\sigma(\gamma)} F_f(\gamma),
\]
where $a^\top(\theta)\from C_\AAA\to \kk$ is defined by
$\langle a^\top(\theta), \gamma\rangle=a(\gamma\theta)$ as in \eqref{eq:xlaw}.

We can now write
\[
B_f(a) = \sum_{\theta \in \T} \frac{S_f(a^\top(\theta))}{\sigma(\theta)} F_f(\theta).
\]
We see that this is simply the expression for a regular B-series, with the
constant coefficient $a(\theta)$ replaced by a $y$-dependent scalar value
$S_f(a^\top(\theta))$.

Now, let $g$ be a smooth function, and expand $g(y+B_f(a)(y))$ using Taylor's theorem.
\[
  g(y+B_f(a)(y)) = g(y) + \sum_{k=1}^{\infty} \frac{1}{k!}g^{(k)}(y)(B_f(a)(y), \dotsc , B_f(a)(y)).
\]
The multi-linearity of $g^{(k)}(y)$ allows us to rewrite the above equation.
Suppressing the dependence of $y$, we get
\begin{align*}
&g(y+B_f(a)(y)) -g(y)=\\
&=\sum_{k=1}^\infty \sum_{(u_1, \dotsc, u_k) \in \T^{k}} \frac{ S_f(a^\top(u_1)) \dotsm S_f(a^\top(u_k))}{k! \sigma(u_1)\dotsm \sigma(u_k)} g^{(k)}(F_f(u_1), \dotsc, F_f(u_k))\\
           &= \sum_{k=1}^\infty \sum_{u_1\dotsm u_k\in \F_k} \frac{ S_f(a^\top(u_1)) \dotsm S_f(a^\top(u_k))}{\mu_1!\dotsm \mu_\nu! \sigma(u_1)\dotsm\sigma(u_k)}F_f(u_1\dotsm u_k)[g].
\end{align*}
In the above equation, $u_1\dotsm u_k$ is a $k$-tuple of rooted trees, $\nu$ is the number of distinct trees among $u_1,\dotsc, u_k$, and $\mu_1,\dotsc,\mu_\nu$ are the multiplicities of each distinct tree.

If $\phi = u_1\dotsm u_m$, the symmetry group $G_\phi$ is the semidirect product of permuting identical trees among $u_1,\dotsc, u_m$ and $G_{u_1}\times \dotsb \times G_{u_m}$, so
\[\sigma(\phi)=\mu_1!\dotsm \mu_\nu! \sigma(u_1)\dotsm\sigma(u_k).\]

We define $\kappa(a)^\top\from H_\F \to C_\AAA^\ast$ to be an algebra morphism by
\[\kappa(a)^\top(\phi)=a^\top(u_1)\dotsm a^\top(u_k), \]
where the product is taken in $C_\AAA^*$.
For $\phi=\efor$, this simplifies to the empty product $\kappa(a)^\top(\efor) = \epsilon_{\AAA}$.

For any $\gamma_1, \gamma_2 \in \AAA$, we have
$F_f(\gamma_1\gamma_2)=F_f(\gamma_1)F_f(\gamma_2)$.
Consequently, we have $S_f(\kappa(a)^\top(\phi))= S_f(a^\top(u_1)) \dotsm S_f(a^\top(u_k))$, and
\[g(y+B_f(a)(y))=\sum_{\phi \in \F}
  \frac{S_f(\kappa(a)^\top(\phi))(y)}{\sigma(\phi)}F_f(\phi)[g](y). \qedhere
\]
\end{proof}

The aromatic S-series $\kappa(a)$ appearing in Theorem \ref{thm:Sser} are, in
general, not characters of $(C_{\AF}, \sqcup)$ with values in $\kk$. To see
this, consider the aromatic graph $\dloop{b}\forest{b}$.
The requirement that $\kappa(a)^\top\from C_{\F} \to  C_{\AAA}^\ast$ is an algebra
morphism implies that
\[\kappa(a)(\dloop{b}) = \langle \kappa(a)^\top(\efor), \dloop{b}\rangle = \langle \epsilon_{\AAA}, \dloop{b} \rangle = 0,\]
on the other hand, $\kappa(a)(\dloop{b}\forest{b})=a(\dloop{b}\forest{b})$ can be nonzero.

The S-series in Theorem \ref{thm:Sser} \emph{are} characters of $H_\F$ with values in $C_\AAA^*$, and such characters form a group when equipped with the convolution product
\[a^\top\cdot b^\top \from \tau \mapsto \sum_{(\tau)}a^\top(\tau_{(1)})\cdot b^\top(\tau_{(2)}).\]
However, the composition product of aromatic S-series described in Theorem \ref{thm: compprod} is a different product.

\section{Substitution law}
\label{sec:sublaw}
In this section, we study the substitution law,
that is, the operation of substituting the vector field $f$ in an aromatic S-series with
another vector field\footnote{or formal sum of vector fields.} $\tilde{f}$
itself expressed as a B-series $\tilde{f}=B_f(b)$.

The combinatorial formulas and algebraic properties for the subsitiution law in
the case of ordinary B-series were described in \cite{CHV10,CEM11}.
In the latter, the substitution law is also described as the dual of a coproduct in a
bialgebra $\tilde{H}$.

In the present paper, only the combinatorial formulas are described for aromatic
B-series.

\begin{definition}
A \emph{partition} $p^\phi$ of an aromatic forest $\phi$ is a partition of the graph $\phi$ into subgraphs $P(p^\phi)=\{\theta_1, \dotsc, \theta_m\}$ such that $\theta_1, \dotsc, \theta_m \in \AT$.
We view $P(p^\phi)$ as a multiset of aromatic trees.
The set of partitions of $\phi$ is denoted $\mathcal{P}(\phi)$.
The \emph{skeleton} $\chi(p^\phi)$ of a partition is the aromatic forest which is obtained by collapsing each of the subgraphs in $P(p^\phi)$ to a single node.
\end{definition}

Constructively, a partition is defined by cutting a subset of the edges in
$\phi$, obtaining a set of connected components, and then adjoining each
rootless connected component, (i.e. aroma) to a connected component with root (i.e. a rooted tree).

\begin{definition}
For the purposes of the proof, we also define a \emph{labeled partition}
$p^{\ast\phi}$ as a partition where we assign each subgraph $\theta_1,\dotsc,
\theta_m$ to a vertex of the skeleton $\chi(p^{\ast\phi})$.
For a labeled partition, we let $P(p^{*\phi})= (\theta_1, \dotsc, \theta_m)$ be
a sequence of aromatic trees.
The set of labeled partitions of $\phi$ is written as $\mathcal{P}^{*}(\phi)$.
\label{def: labelpart}
\end{definition}
For a partition $p^\phi$, the number of possible labellings is
$\sigma(\chi(p^\phi))$.

We illustrate partitions and labeled partitions with an example.
\begin{example}
Let $\phi$ be the aromatic graph:
\[\begin{tikzpicture}
\begin{scope}[etree, scale = 1.5]

\placeroots{3}
\children[1]{child{node(i){}}
             child{node(j){}}}
\children[3]{child{node(k){}}}
  \jointrees{1}{2}
  \end{scope}
  \node[above] at (i) {$1$};
  \node[above] at (j) {$2$};
  \node[above] at (k) {$3$};
  \node[below] at (tree1) {$4$};
  \node[below] at (tree2) {$5$};
  \node[below] at (tree3) {$6$};
\end{tikzpicture}
\]

\begin{enumerate}[(i)]

\item Cutting the edge $4\to 5$ leaves the connected subgraphs
\[\begin{tikzpicture}
\begin{scope}[etree, scale = 1]
  \placeroots{1}
  \children[1]{child{node(n1){}}
               child{node(n2){}}
               child{node(n5){}}}
         \end{scope}
  \node[above] at (n1) {$1$};
  \node[above] at (n2) {$2$};
  \node[above] at (n5) {$5$};
  \node(base)[below] at (tree1) {$4$};
  \node[below=1ex] at (base) {$\theta_1$};
\end{tikzpicture}
\qquad
\begin{tikzpicture}
  \begin{scope}[etree, scale=1]
    \placeroots{1}
    \children[1]{child{node(n3){}}}
  \end{scope}
  \node[above] at (n3) {$3$};
  \node(base)[below] at (tree1) {$6$};
  \node[below=1ex] at (base) {$\theta_2$};
\end{tikzpicture},
\]
which both are trees.
The skeleton is
\[
\begin{tikzpicture}
  \begin{scope}[etree, scale=1]
  \placeroots{2}
  \jointrees{1}{1}
  \end{scope}
  \node[above] at (tree1) {$\theta_1$};
  \node[below] at (tree2) {$\theta_2$};
\end{tikzpicture}.
\]

\item Cutting the edge $2 \to 4$ leaves the connected subgraphs
\[
\begin{tikzpicture}
\begin{scope}[etree, scale = 1]
\placeroots{2}
\children[1]{child{node(n1){}}}
\jointrees{1}{2}
  \end{scope}
  \node[above] at (n1) {$1$};
  \node(base)[above left] at (tree1) {$4$};
  \node[above] at (tree2) {$5$};
  \node[below=2ex] at (tree1) {$\lambda$};
\end{tikzpicture}
\qquad
\begin{tikzpicture}
\begin{scope}[etree, scale=1]
\placeroots{1}
\end{scope}
\node[above] at (tree1) {$2$};
\node[below=2ex] at (tree1) {$\tau_1$};
\end{tikzpicture}
\quad
\begin{tikzpicture}
\begin{scope}[etree, scale=1]
\placeroots{1}
\children[1]{child{node(n3){}}}
\end{scope}
\node[above] at (n3) {$3$};
\node[below] at (tree1) {$6$};
\node[below=2ex] at (tree1) {$\tau_2$};
\end{tikzpicture}.
\]
The connected components are two trees and an aroma, and there is a choice in
which rooted tree to adjoin the aroma to.
If we choose $P(\phi)=\{\lambda \tau_1, \tau_2\}$,
then the skeleton is
\[
\begin{tikzpicture}
  \begin{scope}[etree, scale=1]
  \placeroots{2}
  \jointrees{1}{1}
  \end{scope}
  \node[above] at (tree1) {$\lambda\tau_1$};
  \node[below] at (tree2) {$\tau_2$};
\end{tikzpicture}.
\]

If we choose $P(\phi)=\{\tau_1, \lambda\tau_2\}$,
then the skeleton is
\[
  \begin{tikzpicture}
    \begin{scope}[etree,scale=1]
      \placeroots{1}
      \children[1]{child{node(lt1){}}}
    \end{scope}
    \node[above] at (lt1) {$\tau_1$};
    \node[below] at (tree1) {$\lambda\tau_2$};
  \end{tikzpicture}.
\]
\end{enumerate}

\end{example}

\begin{theorem}
Let $b\from \AT \to \kk$ be an aromatic B-series and $a\from \AF \to \kk$ be an aromatic S-series, then $S_{B_f(b)}(a)$ can be expressed as an aromatic B-series in $f$,  $S_{B_f(b)}(a)=S_f(b\star a)$ with coefficients
\begin{equation}
b\star a(\phi) =\sum_{p^\phi\in \mathcal{P}(\phi)} a(\chi(p^\phi))  \prod_{\theta\in S(p^\phi)}b(\theta).
\label{eq:subprod}
\end{equation}
\label{thm: sublaw}
\end{theorem}

We prove the theorem for $a=\gamma^\ast$ where $\gamma$ is an arbitrary aromatic
forest with $|\gamma|=m$.
The full theorem follows from linearity in $a$.

Let $\tilde{f}=B_f(b)$.
The method used in the proof consists of writing $F_{\tilde{f}}(\gamma)$ as an
$m$-linear function $\EE_\gamma$ evaluated at $m$ copies of $\tilde{f}$ and its derivatives.

This function is then expanded by writing $\tilde{f}=B_f(b)$ as a
sum over $\AT$.
The result is an $m$-tuple sum over aromatic
trees $\AT$.
The terms in this sum are of the form
\[\EE_\gamma(F_f(\theta_1), F_f(\theta_2), \dotsc, F_f(\theta_m))
\]
where $\theta_1, \dotsc,\theta_m$ are aromatic trees, and $\EE_\gamma$ specifies
how derivatives of its arguments shall be combined to form a differential
operator.
This resulting differential operator $\EE_\gamma(F_f(\theta_1), F_f(\theta_2),
\dotsc, F_f(\theta_m))$ is eventually an affine equivariant function of $f$ and
a finite linear combination of elementary differentials of the form $F_f(\phi)$.

Then equal terms in the entire sum are collected.
Proving equality to $S_f(b\star a)$ is done by an application of the
orbit-isotropy theorem.

\begin{proof}
By the definition of S-series (Definition \ref{def:Sseries})
$S_{\tilde{f}}(a) =\frac{1}{\sigma(\gamma)}F_{\tilde{f}}(\gamma)$,
and, by the definition of elementary differential operator (Definition
\ref{def:elemdiff})
\[F_{\tilde{f}}(\gamma)=\EE_{\gamma}(\tilde{f}, \tilde{f}, \dotsc,
  \tilde{f})=\prod_{j\in V(\gamma)} \tilde{f}^{i_j}_{I_{\pi(j)}} \prod_{k\in
    r(\gamma)} \partial_k,\]
where Einstein's summation convention is used.

Now,
\[\tilde{f}=\sum_{\theta\in \AT} \frac{b(\theta)}{\sigma(\theta)}F_f(\theta),\]
so
\[\begin{aligned}
    F_{\tilde{f}}(\gamma)&=\sum_{\theta_1} \sum_{\theta_2}\dotsi
    \sum_{\theta_m}\frac{b(\theta_1)b(\theta_2)\dotsm
      b(\theta_m)}{\sigma(\theta_1)\sigma(\theta_2)\dotsm \sigma(\theta_m)}
    \EE_\gamma(F_f(\theta_1), F_f(\theta_2), \dots, F_f(\theta_m))\\
     &=\sum_{\theta_1} \dotsi
    \sum_{\theta_m}\prod_{j\in V(\gamma)} \frac{b(\theta_j)}{\sigma(\theta_j)}F_f(\theta_j)^{i_j}_{I_{\pi(j)}}\prod_{k\in r(\gamma)}\partial_k,
   \end{aligned}
 \]
 where each of the sums goes over $\theta_k\in \AT$.
In the above equation, each of $F_f(\theta_j)^{i_j}_{I_{\pi(j)}}$ corresponds to taking
derivatives of the elementary differential operators $F_f(\theta)$.

These derivatives are then combined according to $\EE_\gamma$.
The resulting expression is a sum of elementary differential operators
corresponding to aromatic forests.

The graphs of these aromatic forests are formed from the graphs $\theta_1, \dotsc, \theta_m$ in the following manner:
For each edge $e_l=(u, v)$ in $E(\gamma)$, choose a vertex in $\theta_{v}$ and add an edge from the root of $\theta_{u}$ to that vertex.
The sum is then taken over all possible choices of vertices.

We can denote the aromatic forest formed by a particular choice of vertices by
\[\Gamma(\theta_1, \dotsc, \theta_m; j_{1}, \dotsc, j_{n}) = \Gamma(\Theta;\mathbbm{j}).\]
where $j_1, j_2, \dotsc j_n$ are the vertices chosen for edges $e_1, \dotsc, e_n \in E(\gamma)$.
Recall that for $e_l =(u,v)$, we are restricted to choosing $j_l$ in $\theta_{v}$.

With this notation,
\[S_{\tilde{f}}(a)= \frac{1}{\sigma(\gamma)} \sum_{\Theta \in \AT^m} \sum_{\mathbbm{j}}\frac{\prod_i b(\theta_i)}{\prod_i \sigma(\theta_i)}  F_f(\Gamma(\Theta; \mathbbm{j})),\]
where the sum $\sum_{\mathbbm{j}}$ goes over all possible choices of vertices.

We now collect equal terms.
Collecting $\mathbbm{j}$ for which $\Gamma(\Theta; \mathbbm{j})=\phi$ are equal, we get
\begin{equation}
S_g(a)= \frac{1}{\sigma(\gamma)} \sum_{\Theta \in \AT^m} \sum_{\phi} M_1(\gamma,\Theta, \phi)\frac{\prod_ib(\theta_i)}{\prod_i \sigma(\theta_i)}  F_f(\phi),
\label{eq:sublaw1}
\end{equation}
where $M_1(\gamma, \Theta, \phi)$ counts the number of $\mathbbm{j}$ such that $\Gamma(\Theta; \mathbbm{j})=\phi$.

What we seek to prove is that $S_g(a)$ is equal to
\[ S_f(b\star a)=\sum_{\phi} \frac{1}{\sigma(\phi)}\sum_{\substack{p^\phi \in \mathcal{P}(\phi)\\ \chi(p^\phi)=\gamma}} \prod_{\theta\in S(p^\phi)} b(\theta)F_f(\phi).\]

We rewrite this by extending the sum over all labeled partitions, and dividing
by $\sigma(\gamma)$ to compensate.
(Recall Definition \ref{def: labelpart})

\[S_f(b\star a)=\sum_{\phi}\frac{1}{\sigma(\phi)\sigma(\gamma)} \sum_{\substack{p^{*\phi} \in \mathcal{P^*}(\phi)\\ \chi(p^{*\phi})=\gamma}} \prod_{\theta\in S(p^{*\phi})} b(\theta)F_f(\phi).\]
Collecting terms for which $P(p^{*\phi})=(\theta_1, \dotsc, \theta_m) = \Theta$, we get
\begin{equation}
B_f(b\star a)=\sum_{\phi}\frac{1}{\sigma(\phi)\sigma(\gamma)} \sum_{\Theta}M_2(\gamma, \Theta, \phi) \prod_i b(\theta_i)F_f(\phi),
\label{eq:sublaw2}
\end{equation}
where $M_2(\gamma, \Theta, \phi)$ counts the number of labeled partitions $p^{*\phi}$ such that $P(p^{*\phi})=\Theta$.

Comparing \eqref{eq:sublaw1} with \eqref{eq:sublaw2}, $S_g(a)=S_f(b\star a)$ provided that
\begin{equation}
  \frac{M_2(\gamma, \Theta, \phi)}{\sigma(\phi)}=\frac{M_1(\gamma, \Theta, \phi)}{\prod_i \sigma(\theta_i)}.
  \label{eq:Meqsublaw}
\end{equation}
We proceed by proving that, when nonzero, either side of \eqref{eq:Meqsublaw}
can be identified with the reciprocal of the cardinality of an isotropy
subgroup.
Furthermore, we will show that the two isotropy subgroups are isomorphic.

We begin by showing that $M_1$ and $M_2$ are nonzero simultaneously.

Let $\gamma, \Theta=(\theta_1, \dotsc, \theta_m), \phi$ be given.
If there is no way of adding edges to the graphs $\theta_1, \dotsc, \theta_m$ that results in $\phi$, there is also no way to obtain $\theta_1, \dotsc, \theta_m$ by cutting edges in $\phi$ and vice versa.
In this case, $M_1=M_2=0$.

If non-zero, we can identify each side of the equation \eqref{eq:Meqsublaw} with the
reciprocal of the cardinality of subgroups of symmetry groups:

(i) $G_\phi$ acts on labeled partitions of $\phi$ in a natural manner.
Let $p^{*\phi}$ be such that $\chi(p^{*\phi})=\gamma$, and $P(p^{*\phi})=(\theta_1, \dotsc, \theta_m)$.
The action of $G_\phi$ conserves both $\chi(p^{*\phi})$ and $P(p^{*\phi})=(\theta_1, \dotsc, \theta_m)$, and all partitions satisfying $P(p^{\ast \phi})= \Theta$ is obtained in this manner.
Therefore, $M_2(\gamma, \Theta, \phi)$ is the size of the orbit of $p^{*\phi}$ under the $G_\phi$ action.

By the orbit-isotropy theorem, $\frac{\sigma(\phi)}{M_2(\gamma, \Theta, \phi)}$ is the size of the isotropy subgroup of $p^{*\phi}$, that is: the subgroup of $G_\phi$ which does not mix the vertex subsets corresponding to $\theta_1, \dotsc, \theta_m$, and in addition, within each vertex subset $V(\theta_i)$, any vertex whose in-degree is higher in $\phi$ than in $\theta_i$, has to be fixed under the subgroup.

(ii) $\mathbbm{j}$ denotes a sequence of vertices in $\theta_1, \dotsc, \theta_m$ with some restrictions,
$G_{\theta_1}\times \dotsb \times G_{\theta_m}$ acts on the set of such sequences by moving the vertices within each subgraph.
Let $\mathbbm{j}$ be such that $\Gamma(\Theta; \mathbbm{j})=\phi$
The action of $G_{\theta_1}\times \dotsb \times G_{\theta_m}$ preserves $\Gamma(\Theta; \mathbbm{j})$, and all sequences satisfying the requirement are obtained in this manner.
Therefore $M_1(\gamma, \Theta, \phi)$ is the size of the orbit of $\mathbbm{j}$  under the action of $G_{\theta_1}\times \dotsb \times G_{\theta_m}$.

By the orbit-isotropy theorem, $\frac{\prod_i \sigma(\theta_i)}{M_1(\gamma, \Theta, \phi)}$ is the size of the isotropy group of $\mathbbm{j}$ under this action, that is the subgroup of
$G_{\theta_1}\times \dotsb \times G_{\theta_m}$ which, in each component preserves $\mathbbm{j}\cap V(\theta_i)$.

Since the vertices in $\mathbbm{j}$ are exactly those which have higher in-degree in $\phi$ than in $\theta_i$, the two subgroups described in (i) and (ii) are isomorphic.

\end{proof}

The composition and substitution laws are two algebraic operations on aromatic
S-series.
The natural question is how they interact.

For B-series, \cite{CHV10} describes the interaction in terms of the operations
themselves.
The same properties hold for aromatic S-series, as summed up in the following proposition.

\begin{proposition}
\begin{enumerate}[(a)]
\item $(C_{\AT}^\ast, \star)$ forms a monoid with identity $\forest{b}^\ast$ which acts on $C_{\AF}^\ast$ by algebra morphisms, in particular, the following identities hold for $a,a_1,a_2 \in C_{\AF}^\ast$, $b,c\in C_{\AT}^*$, $\alpha_1,\alpha_2 \in \kk$.
\begin{equation}
\begin{aligned}
b\star(\alpha_1 a_1 + \alpha_1 a_1) &= \alpha_1 b \star a_1 + \alpha_2 b\star a_1,\\
c \star (b\star a) &= (c\star b)\star a,\\
b\star \forest{b}^\ast&= b,\\
\forest{b}^\ast \star a &= a,\\
b\star (a_1 \cdot a_2) &= (b\star a_1)\cdot (b\star a_2),\\
b\star a^{-1} &= (b\star a)^{-1}.
\end{aligned}
\end{equation}
\item The set $\{b\in C_{\AT}^\ast \text{ s.t. } b(\forest{b})\neq 0\}$ with operation $\star$ forms a group.
\end{enumerate}
\end{proposition}

\begin{proof}
\begin{enumerate}[(a)]
\item All of the formulas follows from $S_f(b\star a)= S_{B_f(b)}(a)$, and
$S_f(a\cdot b)= S_f(b)\circ S_f(a)$.
\item From (a) we know that $\forest{b}^\ast$ is both a right and left identity on $C_{\AT}^\ast$.
The right inverse of an element $b\in  C_{\AT}^\ast$, $b(\forest{b}) \neq 0$ is given by
\[ \begin{aligned}b^{\star-1}(\forest{b}) &= \frac{1}{ b(\forest{b})}, \\
b^{\star-1}(\tau)&=- \frac{1}{b(\forest{b})}\sum_{\substack{p^\phi\in \mathcal{P}(\phi)\\ p^\phi \neq \text{trivial}}} b(\chi(p^\phi))  \prod_{\theta\in P(p^\phi)}b^{\star-1}(\theta),
\end{aligned}
\]
where in the last sum, the trivial partition with $\chi(p^\phi)=\forest{b}$, $P(p^\phi)=\tau$ is omitted.
All elements in $\{b\in C_{\AT}^\ast \text{ s.t. } b(\forest{b})\neq 0\}$ allow a right inverse in the same set, therefore a right inverse is also a left inverse. \end{enumerate}
\end{proof}

In the nonaromatic case, the interaction between the composition and
substitution products was described on the bialgebraic side via two bialgebras
interacting.

The description of the substitution law as the dual of a coproduct in a fitting
bialgebra and the
corresponding interaction between $C_{\AT}$ and this bialgebra remains unsolved.

\subsection{Divergence of aromatic B-series}
A special case of the substitution law is the divergence of an aromatic
B-series.
Recall that $F_f(\dloop{b})=\tr{f'}=\ddiv f$,
so the divergence of $f$ is given by $S_f(\dloop{b}^*)=F_f(\dloop{b})$.

Let $\tilde{f}=B_f(b)$ be an aromatic B-series.
By Theorem \ref{thm: sublaw} the divergence of $\tilde{f}$ is given by
\[
\begin{aligned}
\ddiv \tilde{f}=& S_{\tilde{f}}(\dloop{b}^*)\\
               =& S_f(b\star \dloop{b}^*),
\end{aligned}
\]
and we define $\nabla b \coloneqq b\star \dloop{b}^\ast$.

By \eqref{eq:subprod}, we have
\[
\nabla b(\phi) =\sum_{p^\phi\in \mathcal{P}(\phi)} \dloop{b}^*(\chi(p^\phi))  \prod_{\theta\in S(p^\phi)}b(\theta).
\]
In this case, we only get non-zero values when $\chi(p^\phi)=\dloop{b}$, thus
$\phi$ can not have any roots, and $P(p^\phi)$ can only consist of a single
aromatic tree.

\[\nabla b(\phi)= \sum_{v \in V(\phi)} b(\phi_v),\]
where $\phi_v$ is the aromatic tree obtained by deleting the edge going out of $v$.
This formula, specialized for B-series, appears in the papers by Iserles,
Quispel and Tse \cite{IQT07}, and the paper by Chartier and Murua \cite{CM07}.

Now, assume that $f$ is a divergence-free vector field $\ddiv f = F_f(\dloop{b})= 0$, and that we apply an aromatic B-series method (i.e. an affine-equivariant, local numerical method) to the differential equation $\dot{y} = f(y)$.
The modified vector field given by $B_f(b)$, and the method is volume-preserving if
$\nabla b  = b\star \dloop{b}^*$ is non-zero only on $\phi \in \AT$ that have $\dloop{b}$ as a subgraph.

\section{Affine-equivariant pseudo-volume-preserving integrators}
\label{sec:final}
One question raised by Munthe-Kaas and Verdier \cite{MV15} is the existence of
volume-preserving aromatic B-series methods.

Volume-preserving methods are well known in the literature, see for instance \cite{quispel95,KS95},
but all such methods rely on some partition of the vector field along coordinate
axes, and are therefore not equivariant under invertible affine maps, and are
therefore not aromatic B-series methods.

It is known that B-series methods  cannot be volume-preserving (apart from the
exact integrator) \cite{IQT07, CM07} but that aromatic B-series methods can be \cite{MV15}.
However, it has so far not been possible to find an equation defining the update map for a volume-preserving aromatic B-series method.

While this quesiton remains unanswered,
with the computational tools developed in this paper, one can at least find
methods that are pseudo-volume-preserving.
That is, that methods that preserves the volume with a higher order of accuracy
than the order of the method itself.

One approach is as follows:

Let $\tilde{f}$ be a preprocessed vector field
\[\tilde{f}=B_f(b),\]
where $B_f(b)$ is an expression that can be efficiently computed from $f$,
for example a linear combination involving $f$, $(f')f$, $(f')^2$,
$\tr\left((f')^2\right)f, \dotsc$.
Then integrate
\[
\dot{x}=\tilde{f}(x)
\]
using a standard Runge--Kutta method or other B-series method with B-series $B_f(a)$.

If $B_f(b)$ is chosen properly, the resulting integrator can preserve the volume
form to higher order than the original B-series method given by $B_f(a)$.

An example of a method which is second order, but preserves volume to fourth order is detailed below.

\begin{example}
\label{exa: pseudo}

We base our method on the implicit midpoint method, whose modified vector field is $B_f(a)$, where
\[a = \forest{b}^* + \frac{1}{12} \left(\forest{b[b[b]]}^*-\forest{b[b,b]}^*\right) + \text{remainder.}\]
Here the remainder contains terms of order 5 and higher.

And define a pre-processed vector field by
\[\tilde{f}(y)=f(y)+\frac{h^2}{12}\left( \frac{1}{2} \tr (f'(y)^2) - f'(y)^2\right)f(y).\]
In the notation of aromatic B-series, $\tilde{f}=B_f(b)$ where
\[b =\forest{b}^*+ \frac{1}{12} \dloop{b,b}\forest{b}^\ast-\frac{1}{12}\forest{b[b[b]]}^\ast.\]

The integrator is defined by substituting $F$ for $f$ in the implicit midpoint method
\begin{equation}y_{n+1} = y_n + hF\left(\frac{y_n+y_{n+1}}{2}\right).
\label{eq:integrator}
\end{equation}
Implementing this method requires evaluating the Jacobian $f'(y)$ and solving a non-linear equation.

The modified vector field of \eqref{eq:integrator} is $B_f(b\star a)$ and the divergence of the modified vector field is $B_f(b\star a \star \dloop{b}^*)$, which we can calculate for the  smallest rootless graphs.
Using \eqref{eq:subprod},
\[b\star a = \forest{b}^*+\frac{1}{12}\left(\dloop{b,b}\forest{b}^*-\forest{b[b,b]}^*\right)+ \text{remainder,}\]
and
\[b\star a \star \dloop{b}^* = \dloop{b}^* -\frac{1}{12} \dloop{b[b,b]}^*+ \text{remainder,}\]
where the remainders contain terms of 5th order or higher.

If $f$ is divergence free, $F_f$ sends any aromatic forest containing the subgraph $\dloop{b}$ to 0,
so we see that while $B_f(b\star a)$ agrees to $f$ to second order, $\ddiv B_f(b\star a)= S_f(b\star a \star \dloop{b}^*)$ disappears to fourth order.
The integrator \eqref{eq:integrator}, while second order, conserves volume to fourth order.

\end{example}

\section*{Acknowledgments}
The research on this paper was partially supported by the the Norwegian Research Council.

The graphs of the aromatic trees were generated in two different ways:
Most of the trees were generated using Håkon Marthinsen's \texttt{planarforest} package \url{http://hmarthinsen.github.io/planarforest/} and the author's modification thereof. These in turn use the \texttt{pythontex}\cite{poore13} package.
The large tree in example \ref{exa: eldiffexample} was generated using Olivier Verdier's \texttt{etrees} package.

I would also like to thank the anonymous referees for many useful comments and
suggestions, and the editors of the issue for the encouragement in finalizing
this article.

\newpage
\appendix
\section{Tables for composition and substitution laws}
\label{app: tables}
In this appendix, we provide tables for the composition and substitution laws on
aromatic B-series.
\begin{table}[h!]
\centering
\begin{tabular}{c|c}
$\phi$ & $\Delta_{\AT}(\phi)$\\
\hline
$\efor$ & $\efor\tensor \efor$ \\
$\forest{b}$ & $\efor\tensor \forest{b} + \forest{b} \tensor \efor$ \\
$\dloop{b}$ & $\efor \tensor \dloop{b} + \dloop{b} \tensor \efor$ \\
$\forest{b[b]}$ & $\efor \tensor \forest{b[b]} + \forest{b} \tensor \forest{b}+ \forest{b[b]} \tensor \efor$ \\
$\dloop{b[b]}$ & $\efor \tensor \dloop{b[b]}+ \forest{b} \tensor \dloop{b} + \dloop{b[b]} \tensor \efor$\\
$\dloop{b,b}$ & $\efor \tensor \dloop{b,b} + \dloop{b,b} \tensor \efor$ \\
$\forest{b[b,b]}$ & $\efor \tensor \forest{b[b,b]} + 2 \forest{b} \tensor \forest{b[b]}+ \forest{b,b}\tensor \forest{b} + \forest{b[b,b]} \tensor \efor$ \\
$\forest{b[b[b]]}$ & $\efor \tensor \forest{b[b[b]]} + \forest{b}\tensor \forest{b[b]}+ \forest{b[b]}\tensor \forest{b}+ \forest{b[b[b]]}\tensor\efor$\\
$\dloop{b[b,b]}$ & $\efor \tensor \dloop{b[b,b]} + 2 \forest{b} \tensor \dloop{b[b]}+ \forest{b,b}\tensor \dloop{b} + \dloop{b[b,b]} \tensor \efor$ \\
$\dloop{b[b[b]]}$ & $\efor \tensor \dloop{b[b[b]]} + \forest{b}\tensor \dloop{b[b]} + \forest{b[b]}\tensor \dloop{b}+ \dloop{b[b[b]]}\tensor\efor$\\
$\dloop{b[b],b}$ & $\efor \tensor \dloop{b[b],b} + \forest{b} \tensor \dloop{b,b} + \dloop{b[b],b} \tensor \efor$\\
$\dloop{b,b,b}$ &  $\efor \tensor \dloop{b,b,b} + \dloop{b,b,b} \tensor \efor$\\
$\forest{b[b,b,b]}$ & $\efor \tensor \forest{b[b,b,b]} + 3\forest{b}\tensor \forest{b[b,b]} + 3 \forest{b,b}\tensor \forest{b[b]}+\forest{b,b,b}\tensor \forest{b} + \forest{b[b,b,b]}\tensor \efor$\\
$\forest{b[b[b],b]}$ & $\forest{b}\tensor\forest{b[b,b]}+ \forest{b}\tensor\forest{b[b[b]]}+ \forest{b[b]}\tensor \forest{b[b]} + \forest{b,b}\tensor \forest{b[b]}+\forest{b[b],b}\tensor \forest{b} + \forest{b[b[b],b]}\tensor \efor$\\
$\forest{b[b[b,b]]}$ &  $\efor \tensor \forest{b[b[b,b]]} + 2 \forest{b} \tensor \forest{b[b[b]]}+\forest{b,b}\tensor \forest{b[b]}+ \forest{b[b,b]}\tensor \forest{b}+ \forest{b[b[b,b]]}\tensor \efor$ \\
$\forest{b[b[b[b]]]}$ & $\efor \tensor \forest{b[b[b[b]]]} + \forest{b}\tensor \forest{b[b[b]]} + \forest{b[b]]}\tensor \forest{b[b]}+\forest{b[b[b]]}\tensor \forest{b} + \forest{b[b[b[b]]]}\tensor \efor$\\
$\dloop{b[b,b,b]}$ & $\efor \tensor \dloop{b[b,b,b]} + 3\forest{b}\tensor \dloop{b[b,b]} + 3 \forest{b,b}\tensor \dloop{b[b]}+\forest{b,b,b}\tensor \dloop{b} + \dloop{b[b,b,b]}\tensor \efor$\\
$\dloop{b[b[b],b]}$ & $\forest{b}\tensor\dloop{b[b,b]}+ \forest{b}\tensor\dloop{b[b[b]]}+ \forest{b[b]}\tensor \dloop{b[b]} + \forest{b,b}\tensor \dloop{b[b]}+\forest{b[b],b}\tensor \dloop{b} + \dloop{b[b[b],b]}\tensor \efor$\\
$\dloop{b[b[b,b]]}$ &  $\efor \tensor \dloop{b[b[b,b]]} + 2 \forest{b} \tensor \dloop{b[b[b]]}+\forest{b,b}\tensor \dloop{b[b]}+ \forest{b[b,b]}\tensor \dloop{b}+ \dloop{b[b[b,b]]}\tensor \efor$ \\
$\dloop{b[b[b[b]]]}$ & $\efor \tensor \dloop{b[b[b[b]]]} + \forest{b}\tensor \dloop{b[b[b]]} + \forest{b[b]]}\tensor \dloop{b[b]}+\forest{b[b[b]]}\tensor \dloop{b} + \dloop{b[b[b[b]]]}\tensor \efor$\\
$\dloop{b[b,b],b}$ & $ \efor \tensor \dloop{b[b,b],b} + 2 \forest{b} \tensor \dloop{b[b],b} + \forest{b,b}\tensor \dloop{b,b}+ \dloop{b[b,b],b}\tensor \efor$\\
$\dloop{b[b[b]],b}$ & $ \efor \tensor \dloop{b[b[b]],b}+\forest{b}\tensor \dloop{b[b],b} + \forest{b[b]}\tensor \dloop{b,b}+ \dloop{b[b[b]],b}\tensor \efor$\\
$\dloop{b[b],b[b]}$ & $\efor \tensor\dloop{b[b],b[b]} + 2\forest{b} \tensor \dloop{b[b],b} + \forest{b,b}\tensor\dloop{b,b}+ \dloop{b[b],b[b]} \tensor \efor$ \\
$\dloop{b[b],b,b}$ & $\efor \tensor \dloop{b[b],b,b} + \forest{b}\tensor \dloop{b,b,b}+\dloop{b[b],b,b}\tensor\efor$\\
$\dloop{b,b,b,b}$ & $\efor \tensor \dloop{b,b,b,b} + \dloop{b,b,b,b} \tensor \efor$
\end{tabular}
\caption{The coproduct $\Delta_{\AT}$}
\label{tab:compprod}
\end{table}

\begin{table}[htp]
\centering
\begin{tabular}{c|c}

$\tau$ & $b\star a(\tau)$\\[1ex]
\hline

  $\forest{b}$ & $a(\forest{b})b(\forest{b})$\\[1ex]

  $\forest{b[b]}$ & $a(\forest{b})b(\forest{b[b]}) +a(\forest{b})b(\forest{b})^2$\\[1ex]

  $\dloop{b}\forest{b}$ & $a(\forest{b})b(\dloop{b}\forest{b})+ a(\dloop{b}\forest{b})b(\forest{b})^2$\\[1.5ex]

  $\forest{b[b,b]}$ & $a(\forest{b})b(\forest{b[b,b]})+2a(\forest{b[b]})b(\forest{b})b(\forest{b[b]})+a(\forest{b[b,b]})b(\forest{b})^3$\\[1ex]

  $\forest{b[b[b]]]}$ & $a(\forest{b})b(\forest{b[b[b]]})+2a(\forest{b[b]})b(\forest{b[b]})b(\forest{b})+a(\forest{b[b[b]]})b(\forest{b})^3$ \\[1.5ex]
  
  $\dloop{b}\forest{b[b]}$ &$a(\forest{b})b(\dloop{b}\forest{b[b]}) +2a(\forest{b[b]})b(\dloop{b}\forest{b})b(\forest{b})+ a(\dloop{b}\forest{b})b(\forest{b[b]})b(\forest{b})+a(\dloop{b}\forest{b[b]})b(\forest{b})^3$\\[1.5ex]
  
  $\dloop{b,b}\forest{b}$ & $a(\forest{b})b(\dloop{b,b}\forest{b})+2a(\dloop{b}\forest{b})b(\forest{b[b]})b(\forest{b})+a(\dloop{b,b}\forest{b})b(\forest{b})^3$\\[1.5ex]

  \multirow{2}{*}{$\dloop{b[b]}\forest{b}$}
  &$a(\forest{b})b(\dloop{b[b]}\forest{b})+ a(\forest{b[b]})b(\dloop{b}\forest{b})b(\forest{b})+a(\dloop{b}\forest{b})b(\dloop{b}\forest{b})b(\forest{b})$\\
  & $+a(\dloop{b}\forest{b})b(\forest{b[b]})b(\forest{b})+a(\dloop{b[b]}\forest{b})b(\forest{b})^3$\\[1.5ex]
  
  $\dloop{b}\dloop{b}\forest{b}$ & $a(\forest{b})b(\dloop{b}\dloop{b}\forest{b})+4a(\dloop{b}\forest{b})b(\dloop{b}\forest{b})b(\forest{b})+a(\dloop{b}\dloop{b}\forest{b})b(\forest{b})^3$\\[1.5ex]

  \multirow{2}{*}{$\dloop{b}\forest{b[b,b]}$}
  &$a(\forest{b})b(\dloop{b}\forest{b[b,b]})+a(\dloop{b}\forest{b})b(\forest{b[b,b]})b(\forest{b})+
   2a(\forest{b[b]})b(\dloop{b}\forest{b})b(\forest{b[b]})+2a(\forest{b[b]})b(\dloop{b}\forest{b[b]})b(\forest{b})$\\
  &$+2a(\dloop{b}\forest{b[b]})b(\forest{b[b]})b(\forest{b})^2+3a(\forest{b[b,b]})b(\dloop{b}\forest{b})b(\forest{b})^2+a(\dloop{b}\forest{b[b,b]})b(\forest{b})^4$\\[1ex]
  
  \multirow{2}{*}{$\dloop{b}\forest{b[b[b]]}$} 
  &$a(\forest{b})b(\dloop{b}\forest{b[b[b]]})+a(\dloop{b}\forest{b})b(\forest{b[b[b]]})b(\forest{b})+2a(\forest{b[b]})b(\dloop{b}\forest{b})b(\forest{b[b]})
   +2a(\forest{b[b]})b(\dloop{b}\forest{b[b]})b(\forest{b})$\\
  &$+2a(\dloop{b}\forest{b[b]})b(\forest{b[b]})b(\forest{b})^2+3a(\forest{b[b[b]]})b(\dloop{b}\forest{b})b(\forest{b})^2+a(\dloop{b}\forest{b[b[b]]})b(\forest{b})^4$\\[1.5ex]
  
  \multirow{2}{*}{$\dloop{b,b}\forest{b[b]}$}
  &$a(\forest{b})b(\dloop{b,b}\forest{b[b]})+2a(\forest{b[b]})b(\dloop{b,b}\forest{b})b(\forest{b})+2a(\dloop{b}\forest{b})b(\forest{b[b]})^2$\\
  &$+2a(\dloop{b}\forest{b[b]})b(\forest{b[b]})b(\forest{b})^2+a(\dloop{b,b}\forest{b})b(\forest{b[b]})b(\forest{b})^2+a(\dloop{b,b}\forest{b[b]})b(\forest{b})^4$\\[1ex]
  
  \multirow{3}{*}{$\dloop{b[b]}\forest{b[b]}$}
  &$a(\forest{b})b(\dloop{b[b]}\forest{b[b]})+a(\forest{b[b]})b(\dloop{b}\forest{b[b]})b(\forest{b})
   +a(\dloop{b}\forest{b})b(\dloop{b}\forest{b})b(\forest{b[b]})+a(\dloop{b}\forest{b})b(\forest{b[b]})b(\forest{b})$\\
  &$+2a(\forest{b[b]})b(\dloop{b[b]}\forest{b})b(\forest{b})+a(\forest{b[b,b]})b(\dloop{b}\forest{b})b(\forest{b})^2+a(\forest{b[b[b]]})b(\dloop{b}\forest{b})b(\forest{b})^2$\\
  &$+a(\dloop{b}\forest{b[b]})b(\dloop{b}\forest{b})b(\forest{b})^2+a(\dloop{b[b]}\forest{b})b(\forest{b[b]})b(\forest{b})^2
    +a(\dloop{b}\forest{b[b]})b(\forest{b[b]})b(\forest{b})^2+a(\dloop{b[b]}\forest{b[b]})b(\forest{b})^4$\\

\end{tabular}
\caption{Substitution law (1 root)}
\label{tab:sublaw1}
\end{table}

\begin{table}[htp]
\centering
\begin{tabular}{c|c}
$\lambda$ & $b\star a(\lambda)$\\[1ex]
  \hline
  
  $\efor$ & $a(\efor)$ \\[1ex]
  
  $\dloop{b}$ & $a(\dloop{b})b(\forest{b})$\\[1ex]
  
  $\dloop{b[b]}$ &$a(\dloop{b})b(\forest{b[b]})+a(\dloop{b})b(\dloop{b}\forest{b})+ a(\dloop{b[b]})b(\forest{b})^2$\\[1ex]
  
  $\dloop{b,b}$ & $2a(\dloop{b})b(\forest{b[b]})+a(\dloop{b,b})b(\forest{b})^2$ \\[1ex]
  
  $\dloop{b}\dloop{b}$ & $2a(\dloop{b})b(\dloop{b}\forest{b})+ a(\dloop{b}\dloop{b})b(\forest{b})^2$ \\[1ex]
  
\multirow{2}{*}{$\dloop{b[b,b]}$} & $2a(\dloop{b})b(\dloop{b[b]}\forest{b})+ a(\dloop{b})b(\forest{b[b,b]})+ 2a(\dloop{b[b]})b(\dloop{b}\forest{b})b(\forest{b})$\\
        & $+ 2a(\dloop{b[b]})b(\forest{b[b]})b(\forest{b})+a(\dloop{b[b,b]})b(\forest{b})^3$ \\[1ex]

\multirow{2}{*}{$\dloop{b[b[b]]}$}  & $ a(\dloop{b})b(\dloop{b[b]}\forest{b})+a(\dloop{b})b(\dloop{b}\forest{b[b]})+a(\dloop{b})b(\forest{b[b[b]]})+a(\dloop{b,b})b(\dloop{b}\forest{b})b(\forest{b})$\\
                   & $+a(\dloop{b[b]})b(\dloop{b}\forest{b})b(\forest{b})+2a(\dloop{b[b]})b(\forest{b[b]})b(\forest{b})+a(\dloop{b[b[b]]})b(\forest{b})^2$\\[1ex]
\multirow{2}{*}{$\dloop{b[b],b}$} & $a(\dloop{b})b(\forest{b[b[b]]})+a(\dloop{b})b(\forest{b[b,b]})+a(\dloop{b})b(\dloop{b,b}\forest{b})$\\
  &$+2a(\dloop{b[b]})b(\forest{b})b(\forest{b[b]})+a(\dloop{b,b})b(\forest{b})b(\forest{b[b]})+a(\dloop{b[b],b})b(\forest{b})^3$\\[1ex]
$\dloop{b,b,b}$ & $3a(\dloop{b}) b(\forest{b[b[b]]})+ 3a(\dloop{b,b})b(\forest{b[b]})b(\forest{b})+a(\dloop{b,b,b})b(\forest{b})^3$ \\[1.5ex]
\multirow{2}{*}{$\dloop{b,b}\dloop{b}$} & $a(\dloop{b})b(\dloop{b,b}\forest{b})+2a(\dloop{b})b(\dloop{b}\forest{b[b]})+2a(\dloop{b}\dloop{b})b(\forest{b[b]})b(\forest{b})$\\
         &$+2a(\dloop{b,b})b(\dloop{b}\forest{b})b(\forest{b})+a(\dloop{b,b}\dloop{b})b(\forest{b})^3$\\[1ex]
\multirow{2}{*}{$\dloop{b[b]}\dloop{b}$} & $a(\dloop{b})b(\dloop{b}\dloop{b}\forest{b})+a(\dloop{b})b(\dloop{b}\forest{b[b]})+a(\dloop{b})b(\dloop{b[b]}\forest{b})+3a(\dloop{b[b]})b(\dloop{b}\forest{b})b(\forest{b})$\\
& $+a(\dloop{b}\dloop{b})b(\forest{b[b]})b(\forest{b})+a(\dloop{b}\dloop{b})b(\dloop{b}\forest{b})b(\forest{b})+a(\dloop{b[b]}\dloop{b})b(\forest{b})^3$\\[1.5ex]
$\dloop{b}\dloop{b}\dloop{b}$ & $3a(\dloop{b})b(\dloop{b}\dloop{b}\forest{b})+ 6a(\dloop{b}\dloop{b})b(\dloop{b}\forest{b})b(\forest{b})+a(\dloop{b}\dloop{b}\dloop{b})b(\forest{b})^3$
\end{tabular}
\caption{Substitution law (0 roots)}
\label{tab:sublaw0}
\end{table}
\FloatBarrier
\section{Coalgebras, bialgebras and Hopf algebras}
\label{app: coalg}
See also \cite{Swe69} for an introduction to coalgebras, bialgebras, Hopf algebras.

Throughout this article, $\kk$ is either $\RR$ or $\CC$, and all tensor products $\tensor$ are taken over $\kk$.
An \emph{associative algebra} over $\kk$ is a vector space $A$ over $\kk$ equipped with an associative linear multiplication map $\mu \from A\tensor A \to A$, $\mu \from a\tensor b \mapsto ab$.
It is \emph{unital} if, in addition, there is a linear map $u\from \kk\to A$, called the \emph{unit} such that for all $\beta\in \kk, a \in A$, $u(\beta)a= au(\beta)=\beta a$.

$A$ is \emph{graded} if $A = \bigoplus_{n=0}^{\infty} A_n$ as a vector space and
$\mu(A_n,A_m) \subset A_{n+m}$.

A \emph{coassociative coalgebra} over $\kk$ is a vector space $C$ equipped with a coproduct $\Delta \from C \to C \tensor C$.
In the notation of Sweedler, we write $\Delta(c)=\sum_{(c)}c_{(1)}\tensor c_{(2)}$.

The coproduct is linear and coassociative, that is
\[(\Delta \tensor I) \circ \Delta = (I \tensor \Delta) \circ \Delta,\]
where $I$ is the identity map on $C$,
or equivalently
\[ \sum_{(c)} \Delta (c_{(1)}) \tensor c_{(2)} = \sum_{(c)} c_{(1)} \tensor \Delta (c_{(2)}). \]

A coalgebra is \emph{counital} if, in addition, there is a linear map $\epsilon \from C \to \kk$, called the \emph{counit} such that for all $c\in C$, $\sum_{(c)} \epsilon( c_{(1)})c_{(2)} = c =  \sum_{(c)} c_{(1)}\epsilon( c_{(2)})$.

The (co)algebras considered in this paper are all both (co)associative and (co){-}unital, and we will omit the qualifiers.

A coalgebra $C$ is \emph{graded} if $C=\bigoplus_{n=0}^\infty C_n$ as a vector space and
$\Delta(C_n) \subset \bigoplus_{k+l=n}C_k\tensor C_l$.
A coalgebra is \emph{connected} if $C_0 \simeq \kk$.

For $V,W$ algebras, respectively coalgebras, $V\tensor W$ is again an algebra,
respectively a coalgebra in a canonical way:
\[(v_1 \tensor w_1)(v_2 \tensor w_2) = v_1v_2 \tensor w_1w_2\]
or
\[\Delta(v\tensor w) = \sum_{(v)}\sum_{(w)} v_{(1)}\tensor w_{(1)}\tensor v_{(2)}\tensor w_{(2)}\]

A \emph{bialgebra} $B$ is simultaneously an algebra and coalgebra such that the maps $\mu \from B \tensor B \to B$ and $u\from \kk \to B$ are morphisms of counital coalgebras and
$\Delta \from B \to B \tensor B$ and $\epsilon \from B \to \kk$ are morphisms of unital algebras, that is, for all $h,k \in B$,
\[
\begin{aligned}
\Delta(hk)&=\sum_{(h)}\sum_{(k)} h_{(1)}k_{(1)}\tensor h_{(2)}k_{(2)} \\
\Delta(u(1)) &= u(1) \tensor u(1) \\
\epsilon(hk) &= \epsilon(h)\epsilon(k) \\
\epsilon(u(1)) &= 1
\end{aligned}
\]

$B$ is graded if $B = \bigoplus_{n=0}^\infty B_n$ is graded as an algebra and coalgebra, and connected if $B_0\simeq \kk$.

For $C$ coalgebra and $A$ algebra, $\Hom_{\kk}(C,A)$ equipped with the \emph{convolution product}
$f\ast g = \mu_{A} \circ (f\tensor g) \circ \Delta_C$
is an algebra.
The unit of the convolution product is $u_A \circ \epsilon_C$.

For $B$ a bialgebra, $\Hom_{\kk}(B,B)$ equipped with the convolution product is an algebra.
If the convolution inverse of the identity map exists, it is called the \emph{antipode}, and in this case the bialgebra is called a \emph{Hopf algebra}.

\begin{theorem}
\label{thm: bihopf}
\cite{Swe69}
A graded, connected bialgebra is a graded Hopf algebra.
\end{theorem}

\printbibliography
\end{document}